\documentclass[9pt,twoside]{amsart}
\usepackage{amsxtra}
\usepackage{color}
\usepackage{amsmath,amsthm,amssymb,bm}
\usepackage{mathrsfs,mathtools}
\usepackage{stmaryrd}
\usepackage{hyperref}
\usepackage{enumerate}
\usepackage{xcolor}
\usepackage{pifont}
\usepackage{adjustbox}
\usepackage{tikz}

\theoremstyle{theorem}
\newtheorem{theorem}{Theorem}[section]
\newtheorem{corollary}[theorem]{Corollary}
\newtheorem{proposition}[theorem]{Proposition}
\newtheorem{lemma}[theorem]{Lemma}
\newtheorem*{theorem*}{Theorem}
\theoremstyle{definition}

\newtheorem{example}[theorem]{Example}

\theoremstyle{remark}
\newtheorem{remark}[theorem]{Remark}

\newcommand{\beq}{\begin{equation}}
\newcommand{\eeq}{\end{equation}}
\renewcommand{\a}{\alpha}
\renewcommand{\b}{\beta}

\newcommand{\cF}{\mathcal{F}}

\newcommand{\cH}{\mathcal{H}}

\newcommand{\cV}{\mathcal{V}}

\newcommand{\U}{{\mathrm U}}

\newcommand{\M}{{\mathrm M}}

\newcommand{\n}{\nabla}

\DeclareMathOperator\tr{tr\;}
\DeclareMathOperator\Tr{Tr}

\DeclareMathOperator\vol{vol}

\DeclareMathOperator{\tor}{Tor\;}

\newcommand{\Ric}{{\rm Ric}}

\newcommand{\Scal}{{\rm Scal}}

\renewcommand{\M}{{\rm M}}

\newcommand{\ric}{{\rm{Ric}}}

\numberwithin{equation}{section}

\title[Lagrangian Foliations on compact K\"ahler manifolds ]{Lagrangian Foliations on compact K\"ahler manifolds}
\dedicatory{Dedicated to the memory of Richard Palais}
\author{Fabio Podest\`a}
\address{Dipartimento di Matematica e Informatica ``U.~Dini'' \\ Universit\`a degli Studi di Firenze\\ Viale Morgagni 67/a\\ 50134 Firenze\\ Italy}
\email{fabio.podesta@unifi.it}
\subjclass[2000]{53C55, 53C12, 53D12}
\keywords{Compact K\"ahler manifolds, Lagrangian submanifolds, Riemannian foliations.}

\begin{document}
\maketitle
\begin{abstract} We prove that a compact K\"ahler manifold carrying a regular Riemannian foliation whose leaves are Lagrangian and minimal is flat. In order to prove this result, we establish a foliated version of an integral formula due to Ros.
\end{abstract}
\section{Introduction}

A regular foliation $\mathcal F$ on a manifold $\M$ is given by the collection of all maximal integral submanifolds of an involutive smooth distribution $\cV$ and when the manifold is equipped with a Riemannian metric, the foliation $\mathcal F$ is called Riemannian if its leaves are locally everywhere equidistant (see e.g. \cite{Molino},\cite{Tondeur},\cite{Rovenski}). 

In this work we focus on the case when the ambient manifold is a compact $2n$-dimensional K\"ahler manifold $(\M,g,J)$ and the leaves of the Riemannian foliation $\mathcal F$ are Lagrangian. Such foliations provide a natural meeting point between symplectic geometry and the geometry of foliations and they have been investigated by several authors, providing a deep understanding of their local (see e.g. \cite{Vaisman},\cite{Weinstein}, \cite{Dazord}) and global structure (see e.g. \cite{Duistermaat},\cite{HK}). 

The simplest examples are given by flat complex tori, endowed with their standard Lagrangian foliations.  The Riemannian and Lagrangian conditions alone, however, are not enough to allow for a classification result and not even to prove that the manifold has to be flat, as it can be easily shown by a simple example (see section 1). The question whether examples other than tori exist was already raised by Vaisman in \cite{Vaisman} and rigidity results have instead been obtained under additional assumptions; in particular, Hamilton and Kotschick (\cite{HK}) proved that a K\"ahler metric admitting a parallel Lagrangian foliation is necessarily flat. Clearly the variety of examples becomes much wider if we allow the foliation to be singular with regular Lagrangian leaves, as it happens in the toric case, when a torus $\operatorname{T}^n$ acts on a compact K\"ahler manifold $\M$ in a Hamiltonian fashion.

The additional condition which restores rigidity turns out to be minimality of the leaves. In this work we investigate this condition of minimality and prove the following result
\begin{theorem} Let $(\M,g,J)$ be a compact K\"ahler manifold and let $\mathcal F$ be a regular Riemannian foliation whose leaves are Lagrangian and minimal. Then $(\M,g)$ is flat. 
\end{theorem}

As already mentioned, a related rigidity result was obtained by Hamilton and Kotschick
\cite{HK}, who proved that a K\"ahler metric admitting a parallel Lagrangian foliation is necessarily flat.  Our result may be viewed as a criterion
forcing parallelism: under compactness, the considerably weaker
assumptions that the Lagrangian foliation is Riemannian and has
minimal leaves are enough to prove that its second fundamental form vanishes.

We now briefly describe the main ingredients of the proof and the outline of the paper.  If $\cH=\cV^\perp=J\cV$ denotes the orthogonal distribution, we consider the fundamental O'Neill tensors $A$ and $T$ associated with the orthogonal splitting
$T\M=\cV\oplus\cH$.  In the first section we gather some basic but fundamental results concerning these tensors, the curvature and the Ricci tensor of $(\M, g)$. A first consequence of the K\"ahler and bundle-like conditions is that $A$ vanishes identically. This implies that $\cH$ is integrable and its leaves are totally geodesic. Therefore the manifold turns out to be bi-Lagrangian (also said para-K\"ahler) (see e.g. \cite{Etayo},\cite{Boyom}, \cite{Cruceanu}). The remaining
extrinsic geometry of the foliation is encoded by the symmetric
cubic tensor
\[
S(U,V,W)=g(T_UV,JW),
\]
where $U,V,W$ are vertical vector fields. Then, if $N$ denotes the mean curvature vector of the leaves of $\mathcal F$, a simple
computation shows that 
\[
\Scal=2\operatorname{div}N.
\]
On the other hand, the Lagrangian splitting
\[
(T\M,J)\simeq \cV\otimes_{\mathbb R}\mathbb C
\]
implies that the structure group reduces from $\U(n)$ to ${\rm{O}}(n)$ and therefore $c_1(\M)=0$ in real cohomology.  Consequently, when the
leaves are minimal, the metric is scalar-flat. Since $\M$ is compact Kähler and $c_1(\M)=0$ in real cohomology, this implies that the Ricci form vanishes, and hence \(g\) is Ricci-flat.

The second section deals with the main new tool that is used in the proof of the main theorem. Indeed, instead of chasing after a Bochner-like integral formula that we were unable to find, we liked to investigate and to adapt an old but powerful integral formula established by Ros (\cite{Ros}) in a paper concerning the geometry of complex submanifolds of the complex projective space.  We consider the unit sphere bundle $U\cH$ of the horizontal distribution together with the horizontal geodesic vector field $G_{\cH}$.  Although
the horizontal leaves are totally geodesic, the natural measure on $U\cH$ is not preserved by $G_{\cH}$ in general. Tangent sphere bundles of foliations and the corresponding leafwise geodesic flows have previously been studied by Rovenski and Walczak \cite{RovenskiWalczak}, who use the analysis of transverse Jacobi tensors to obtain integral formulae involving the conullity and mixed curvature operators. In our setting, their first variation of the transverse Jacobian is closely related to the identity 
\[
\operatorname{div}_{d\nu_{\cH}}G_{\cH}(p,v)
=-g(N_p,v)\qquad (p,v)\in U\cH.
\]
that we independently prove for the sake of completeness. Thus minimality is precisely the condition which makes the horizontal geodesic flow measure preserving.  Similarly to what happens in Ros' formula, we have that, for every
covariant tensor $Q$ on $\cV$,
\[
\int_{U\cH}
(\bar\nabla_vQ)(Jv,\ldots,Jv)\,d\nu_{\cH}=0,
\]
where $\bar\n$ is an adapted linear Hermitian connection leaving $\cV$ parallel.

There is a classical literature on integral formulae for
Riemannian foliations, involving the mean curvature, the second
fundamental form and the mixed scalar curvature; see, for instance,
\cite{Walczak} and the references therein. We are not aware of this version in the existing literature.

In the third section we apply this formula to a suitable tensor constructed from $S$ and its
horizontal covariant derivative, noting that we need an odd-degree tensor if we like the formula to be effective.  More precisely, setting for $(U,V,W,Z\in \Gamma(\cV)$
\[
C(U,V,W,Z)=(\bar\n_{JU}S)(V,W,Z).
\]

Ricci-flatness implies that $C$ is a totally symmetric trace-free
$4$-tensor.  Applying the above integral formula to the
symmetrization of $S\otimes C$, and performing the integration on the
unit spheres, a non-trivial algebraic computation leads to
\[
0=
\int_M\left(
|C|^2+
\left|\sum_i B_i^2\right|^2+
\sum_{i,j}|[B_i,B_j]|^2
\right)d\mu_M,
\]
where the operators $B_i$, $i=1,\ldots,n$, are suitable symmetric operators related to the tensor $T$.  Every term in the integrand is non-negative, and
the identity therefore forces $B=0$, equivalently $T=0$.  The
foliation is consequently parallel, and the flatness of $M$ follows.

The author does not know whether the hypothesis of compactness can be relaxed to completeness, as in this case the involved techniques would be totally different, but it seems unlikely that the same rigidity can be obtained from purely local considerations; the proof presented here makes essential use of compactness.\par\medskip\vspace{1cm}

\paragraph{\bf{Acknowledgments.}} The author thanks Marco Radeschi and Leonardo Biliotti for valuable conversations and he also acknowledges the use of ChatGPT-5.6 Sol for assistance in identifying formula \eqref{sphereint} in the existing literature and for double-checking the calculations in this work. The author is partially supported by GNSAGA (INdAM, Italy).

\par\medskip\vspace{2cm}
\section{Preliminaries}
We consider a compact K\"ahler manifold $(\M,g,J)$ of real dimension $2n$, which is endowed with an involutive smooth distribution $\cV$ whose maximal integral submanifolds are Lagrangian. We will say that the foliation $\mathcal F$ given by the union of all these submanifolds is a Lagrangian foliation. We will also assume that the foliation is Riemannian, namely that the metric is bundle-like, i.e.
\beq\label{bundle}\mathcal L_Vg(X,Y) = 0,\eeq
for every vertical vector fields $V\in\Gamma(\cV)$ and horizontal vector fields $X,Y\in\Gamma(\cH)$, where $\cH := \cV^\perp$.\par
The standard examples of such manifolds and Lagrangian foliations are provided by the standard torus $\operatorname{T}^{2n}=\mathbb R^{2n}/\mathbb Z^{2n}$ endowed with the standard flat K\"ahler metric $\omega=\sum_{i=1}^ndx_i\wedge dy_i$ so that the Lagrangian leaves are given by the slices $\{x_i=c_i,\ i=1,\ldots,n\}$. It is also not difficult to exhibit non-flat examples.
\begin{example} On $\operatorname{T}^2$ we construct the metric $g=dx_1^2+f(x_1)^2dy_1^2$ for some positive periodic function $f\in C^\infty(\mathbb R)$ and we define a complex structure $J_f$ with $J(\frac\partial{\partial x_1})= \frac 1{f(x_1)}\frac\partial{\partial y_1}$. Then the metric $g$ is non-flat whenever $f''\not\equiv 0$ and $\{x_1=c\}$ are trivially Lagrangian lines. This can be generalized on $\operatorname{T}^{2n}= \operatorname{T}^2\times \operatorname{T}^{2n-2}$ by taking the product with the standard $\operatorname{T}^{2n-2}$.
\end{example}

We now recall the definition of the basic tensors $A$ and $T$ (see \cite{Bes},\cite{Molino}). For $V,W\in\Gamma(T\M)$ we have by definition 
$$T_VW= (\nabla_{V^\cV}W^\cV)^\cH + (\nabla_{V^\cV}W^\cH)^\cV,$$
$$A_VW = (\nabla_{V^\cH}W^\cV)^\cH + (\nabla_{V^\cH}W^\cH)^\cV,$$
where $\n$ denotes the Levi Civita connection and the superscripts ${}^\cV$ and ${}^\cH$ denote the vertical and horizontal components along $\cV$ and $\cH$ resp.
The condition \eqref{bundle} can be then equivalently written as
\beq\label{bundle2}A_XY + A_YX = 0, \qquad X,Y\in \Gamma(\cH)\eeq
while the K\"ahler condition $\n J=0$ gives for $V\in \Gamma(\cV)$ and $X,Y\in\Gamma(\cH)$
\beq\label{K}A_XJY = JA_XY, \qquad T_VJX=JT_VX.\eeq
For every $V\in \Gamma(\cV)$ we have that $T_V$ is a skew-symmetric endomorphism of $TM$ mapping vertical vectors to horizonal vectors and viceversa. Therefore it is convenient to define
$$B_V:= -T_V\circ J,\quad V\in\Gamma(\cV)$$ 
which is a symmetric endomorphism preserving vertical and horizontal parts. 
It then follows that the $(0,3)$-tensor $S$ on $\cV$ given by
$$S(U,V,W) = g(T_UV,JW)= g(B_UV,W),\qquad U,V,W\in\Gamma(\cV)$$
is totally symmetric, hence a section of $S^3(\cV^*)$. We now prove the first result
\begin{proposition} The tensor $A$ vanishes identically, hence the orthogonal distribution $\cH$ is integrable with totally geodesic leaves.
\end{proposition}
\begin{proof} Let $X,Y,Z$ be horizontal vector fields. Then
$$\phi(X,Y,Z):= g(A_XY,JZ) = g(\nabla_XY,JZ)= -g(\nabla_XJY,Z) = $$
$$= -Xg(JY,Z) + g(JY,\nabla_XZ) = \phi(X,Z,Y).$$
On the other hand we have $\phi(X,Y,Z)= - \phi(Y,X,Z)$. Hence
$$\phi(X,Y,Z)=-\phi(Y,X,Z)=-\phi(Y,Z,X)= \phi(Z,Y,X)=$$
$$= \phi(Z,X,Y)=-\phi(X,Z,Y) = -\phi(X,Y,Z) = 0.$$
It follows that $A_XY=0$ and therefore $A\equiv 0$ by \eqref{K}. Note that the argument involving $\phi$ is equivalent to the well-known fact that the partial symetrization map $c:\Lambda^2V^*\otimes V^*\to V^*\otimes \Lambda^2V^*$ is injective, where $V$ is any vector space.\end{proof}
\begin{remark} Note that the leaves of $\cH$ are Lagrangian, but the corresponding foliation is not necessarily Riemannian; in this case the same argument would force $\cF$ to be totally geodesic and therefore parallel.\end{remark}
\begin{remark}
Since $\cH$ is integrable and both $\cV$ and $\cH$ are Lagrangian,
$(M,\omega,\cV,\cH)$ is a bilagrangian manifold.  In this setting
one may consider the canonical symplectic connection associated with
the two Lagrangian foliations.  It is worth mentioning that, by a
result of Vaisman (\cite{Vaisman2}), the minimality assumption implies that this
canonical connection is Ricci-flat.  We shall not use this fact in
the sequel.
\end{remark}

Using well-known formulas for the curvature (see e.g. \cite{Bes}, \cite{Tondeur}) we have some implications for the curvature tensor $R(X,Y)=[\n_X,\n_Y]-\n_{[X,Y]}$. Before summerizing these in the following Lemma, we introduce the adapted connection $\bar\n$ that is defined as 
\[
\bar\nabla_E G
:=(\nabla_E G^{\mathcal V})^{\mathcal V}
 +(\nabla_E G^{\mathcal H})^{\mathcal H}.
\]
for every vector fields $E,G\in\Gamma(TM)$. It is clear that $\bar\n$ preserves $\cV$ and $\cH$ and it is complex Riemannian, namely $\bar\n g=0$ and $\bar\n J=0$.

\begin{lemma}\label{curv}We have 
\begin{itemize}
\item[i)]  $\label{f1}g(R(X,Y)Z,V) = 0$ for $X,Y,Z\in \Gamma(\cH),\ V\in\Gamma(\cV)$;
\item[ii)] for every $U,V\in\Gamma(\cV)$ we have $R(U,V)(\cV)\subseteq \cV$ and 
\beq\label{f4}R(U,V)|_{\mathcal V} = [T_U,T_V] = - [B_U,B_V]\eeq
\item[iii)] for $X,Y\in \Gamma(\cH)$ and $U,V\in\Gamma(\cV)$ we have
\begin{equation}\label{mixed-curvature}
g(R(X,U)V,Y)
=
g((\bar\nabla_XT)_UV,Y)
+g(T_{T_UX}V,Y);
\end{equation}
\item[iv)] the leaves of $\mathcal F$ are flat.
\end{itemize}
\end{lemma} 
\begin{proof} Formula (i) follows form \cite{Bes}, (9.28e) (beware the curvature in \cite{Bes} has the opposite sign), while the first statement of (ii) follows form (i) and the K\"ahler identities. Moreover using \cite{Bes}, Thm. 9.28, we have for $V,W\in\Gamma(\cV)$ and $X,Y\in\Gamma(\cH)$
$$g(R(V,W)X,Y) = g(T_VX,T_WY) - g(T_WX,T_VY)$$ 
and therefore if $V'=JX,W'=JY$
$$g(R(V,W)V',W') = g( T_V(-JV'),T_W(-JW')) - g(T_W(-JV'),T_V(-JW'))  $$
$$= g(T_VV',T_WW') - g(T_WV',T_VW') = g([T_V,T_W]V',W').$$ 
As for (iii), extend $X,Y$ as basic horizontal vector fields. Since
$A=0$, differentiation in a horizontal direction preserves both
$\mathcal V$ and $\mathcal H$, and since $X$ is basic,
$[X,U]$ is vertical. We have 

\begin{align*}
g(R(X,U)V,Y)
&=g(\bar\nabla_X(T_UV),Y)
  -g(T_U(\bar\nabla_XV),Y)
  -g(T_{[X,U]}V,Y)\\
&=g((\bar\nabla_XT)_UV,Y)
  +g(T_{\bar\nabla_XU-[X,U]}V,Y).
\end{align*}
Since
\[
\bar\nabla_XU-[X,U]=\nabla_UX=T_UX,
\]
we have
\begin{equation}\label{mixed-curvature}
g(R(X,U)V,Y)
=
g((\bar\nabla_XT)_UV,Y)
+g(T_{T_UX}V,Y).
\end{equation}

We then have that (iv) follows from \cite{Bes}, (9.28a).\end{proof}
\par\bigskip
\subsection{The Ricci tensor.} The next proposition states some facts on the Ricci tensor of $(\M,g)$ (see \cite{Bes}, \cite{Tondeur} for general formulas). It will be useful to denote by $N$ the mean curvature vector of the leaves of $\mathcal F$, namely
$$N := \sum_{i=1}^nT_{V_i}V_i,$$
where $\{V_i\}_{i=1,\ldots,n}$ is an o.n. basis in $\cV$.
\begin{proposition} We have 
\begin{itemize}
\item[i)] $c_1(M)=0\in H^2(\M,\mathbb R)$;
\item[ii)] $\ric(\cV,\cH)=0$;
\item[iii)] if $U,V\in \Gamma(\cV)$ 
we have \begin{equation}\label{ricci-main}
\Ric(U,V)
=
\sum_i g((\bar\nabla_{X_i}T)_UV,X_i)
-
g(T_UV,N);\end{equation}
\item[iv)] we have $$\operatorname{scal} = 2\operatorname{div}N.$$ 
\end{itemize}
\end{proposition}
\begin{proof} As for (i), since $\mathcal F$ is Lagrangian, we have $TM=\mathcal V\oplus J\mathcal V$ and in particular, as a complex vector bundle, $(TM,J)\simeq \mathcal V\otimes_{\mathbb R}\mathbb C$. Thus the structure group of $(TM,J)$ reduces from $U(n)$ to $O(n)$. For every real vector bundle $E$, one has
$\overline{E_{\mathbb C}}\simeq E_{\mathbb C}$ 
and therefore
\[
c_k(E_{\mathbb C})
 =c_k(\overline{E_{\mathbb C}})
 =(-1)^k c_k(E_{\mathbb C}).
\]
It follows that $2c_{k}(TM,J)=0$ for all odd $k$. In particular $c_1(M) = 0 \in H^2(M,\mathbb R)$. This fact was also noticed in \cite{Vaisman}.\par
As for (ii), let $U,W\in\Gamma(\mathcal V)$ and put $X=JW\in\Gamma(\mathcal H)$, so that 

\[
\Ric(U,X)
=
\sum_i R(V_i,U,X,V_i)
+
\sum_i R(X_i,U,X,X_i).
\]
The first term vanishes by Lemma \ref{curv}, (ii).  For the second term, using $R(A,B)J=JR(A,B)$ and $JX_i=-V_i$, we have
\begin{align*}
R(X_i,U,JW,X_i)
&=-R(X_i,U,W,JX_i)
=R(X_i,U,W,V_i)\\
&=R(W,V_i,X_i,U)=-R(W,V_i,U,X_i)=0,
\end{align*}
where in the last equality we have again used Lemma \ref{curv}, (ii). Hence $\ric(U,X)=0$.\par
As for (iii), let $\{V_1,\ldots,V_n\}$ be a local orthonormal frame of $\mathcal V$ and put
$X_i=JV_i$. For $U,V\in\Gamma(\mathcal V)$ we have
\[
\Ric(U,V)
=
\sum_i g(R(V_i,U)V,V_i)
+
\sum_i g(R(X_i,U)V,X_i).
\]

We first consider the vertical contribution. By \eqref{f4},
\[
R(V_i,U)|_{\mathcal V}
=
-[B_{V_i},B_U],
\]
and therefore
\begin{align*}
\sum_i g(R(V_i,U)V,V_i)
&=
-\sum_i g([B_{V_i},B_U]V,V_i)\\
&=
-\sum_i g(B_{V_i}B_UV,V_i)
+\sum_i g(B_UB_{V_i}V,V_i).
\end{align*}
If
\[
\eta:=\sum_iB_{V_i}V_i=-JN,
\]
then, using the total symmetry of $S$,
\[
-\sum_i g(B_{V_i}B_UV,V_i)
=
-g(B_UV,\eta)
=
-g(T_UV,N),
\]
whereas
\[
\sum_i g(B_UB_{V_i}V,V_i)
=
\sum_i g(B_UV_i,B_VV_i).
\]
Hence
\begin{equation}\label{ricci-vertical-part}
\sum_i g(R(V_i,U)V,V_i)
=
-g(T_UV,N)
+
\sum_i g(B_UV_i,B_VV_i).
\end{equation}

For the horizontal contribution we use \eqref{mixed-curvature}:
\begin{align*}
\sum_i g(R(X_i,U)V,X_i)
&=
\sum_i g((\bar\nabla_{X_i}T)_UV,X_i)
+
\sum_i g(T_{T_UX_i}V,X_i).
\end{align*}
Since
\[
T_UX_i=T_U(JV_i)=-B_UV_i,
\]
the total symmetry of $S$ gives
$$
\sum_i g(T_{T_UX_i}V,X_i)=
-\sum_i S(B_UV_i,V,V_i)=
-\sum_i g(B_UV_i,B_VV_i).$$
Therefore
\begin{equation}\label{ricci-horizontal-part}
\sum_i g(R(X_i,U)V,X_i)
=
\sum_i g((\bar\nabla_{X_i}T)_UV,X_i)
-
\sum_i g(B_UV_i,B_VV_i).
\end{equation}

Adding \eqref{ricci-vertical-part} and
\eqref{ricci-horizontal-part}, the quadratic terms cancel and we obtain
\eqref{ricci-main}. \par
As for (iv) we have 
\begin{align*}
\sum_j\Ric(V_j,V_j)
&=
\sum_{i,j}
g((\bar\nabla_{X_i}T)_{V_j}V_j,X_i)
-|N|^2\\
&=
\sum_i g(\bar\nabla_{X_i}N,X_i)-|N|^2=
\operatorname{div}N
\end{align*}
and the claim follows using the $J$-invariance of the Ricci tensor. \end{proof}
\begin{corollary} If all the leaves of $\mathcal F$ are minimal, then $(\M,g)$ is Ricci flat.
\end{corollary}
\begin{proof} Indeed in this case $N=0$ and $M$ is scalar flat. We already know that $c_1(M)=0$, i.e. the Ricci form is $dd^c\varphi$ for some smooth function $\varphi$. As the scalar curvature vanishes, the function $\varphi$ is harmonic, hence constant and our claim follows. \end{proof}

\par\bigskip

\section{A Ros' formula for the horizontal foliation and the proof of the main theorem.}

We now like to adapt an integral formula used by Ros in \cite{Ros} to our setting. In \cite{Ros} the following Lemma is stated and proved 
\begin{lemma} Let $(\M,g)$ be a compact Riemannian manifold and let $T$ be a $k$-covariant tensor. Then 
$$\int_{UM} \nabla T(u,\ldots,u) d\nu =0,$$
where $d\nu$ is the canonical measure on the unit tangent bundle $UM$ of $\M$.\end{lemma} 
The proof of this powerful tool is provided using algebraic methods, in particular with Weyl's theorem on invariants. A more geometric proof can be given using the vector field $G_{UM}$ on $UM$ that induces the geodesic flow. Indeed, by Liouville's theorem $\mathcal L_{G_{UM}}d\nu = 0$ (see e.g. \cite[Lemma~3.6.4]{PSU}) and moreover if $\tau\in C^\infty(UM)$ is the map given by $\tau(u)=T(u,\ldots,u)$, then $(G_{UM}\tau)_u = \nabla_uT(u,\ldots,u)$ (see e.g. \cite[Lemma~2.11]{Lef}), so that the claim in Ros' Lemma is proved using an integration by parts argument.

We now consider the horizontal distribution $\cH$. Since $A=0$, the
distribution $\cH$ is integrable and its leaves are totally geodesic submanifolds of $M$.

Let
\[
U\cH:=\{(p,X)\in\cH:\ |X|=1\}.
\]
We endow $U\cH$ with the natural measure $d\nu_{\cH}$ obtained by
integrating the standard spherical measure on $U\cH_p$ against the
Riemannian measure of $M$: for every $f\in C^\infty(U\cH)$
\[
\int_{U\cH}f\,d\nu_{\cH}
=
\int_M
\left(
\int_{U\cH_p}f(p,v)\,d\sigma_p(v)
\right)d\mu_M(p).
\]

Let $G_{\cH}$ denote the geodesic vector field of the horizontal
foliation. Thus, if $(p,X)\in U\cH$ and $\gamma$ is the geodesic in
the horizontal leaf through $p$ such that $\gamma(0)=p,\dot\gamma(0)=X$, 
then
\[
G_{\cH}(p,X)
=
\left.\frac{d}{dt}\right|_{t=0}
(\gamma(t),\dot\gamma(t)).
\]
Since the horizontal leaves are totally geodesic, $\gamma$ is also an
ambient geodesic, hence $\n_{\dot\gamma}\dot\gamma=0$. Moreover $G_\cH$ is tangent to $U\cH$ as the metric is bundle-like and, is $\pi:U\cH\to M$ denotes the projection, then 
$\pi_*G_\cH|_{(p,X)}=X$ for every $(p,X)\in U\cH$.

We now compute explicitly the divergence of $G_{\cH}$ with
respect to $d\nu_{\cH}$.

\begin{lemma}\label{horizontalLiouville}
For $(p,X)\in U\cH$ one has
\beq\label{horizontaldivergence}
\operatorname{div}_{d\nu_{\cH}}G_{\cH}(p,X)
=
-g(N,X).
\eeq
In particular, if the leaves of $\cF$ are minimal, then
\[
\operatorname{div}_{d\nu_{\cH}}G_{\cH}=0.
\]
\end{lemma}

\begin{proof}
Since both $\cV$ and $\cH$ are integrable, locally we can choose
coordinates $(x^1,\ldots,x^n,y^1,\ldots,y^n)$ 
such that the $x$-directions are horizontal and the $y$-directions
are vertical. As the two distributions are orthogonal, the metric
has the form
\[
g
=
h_{ij}(x,y)\,dx^i dx^j
+
k_{\alpha\beta}(x,y)\,dy^\alpha dy^\beta.
\]

The original foliation $\cF$ is Riemannian, hence its transverse
metric is basic. Therefore
\[
\frac{\partial h_{ij}}{\partial y^\alpha}=0,
\]
so that $h_{ij}=h_{ij}(x)$. 
Hence
\beq\label{volumeproducthorizontal}
d\mu_M
=
\sqrt{\det h(x)}
\sqrt{\det k(x,y)}
\,dx\,dy.
\eeq

For fixed $y$, the corresponding horizontal leaf has Riemannian
metric $h_{ij}(x)\,dx^i dx^j$. 
Let $d\lambda_{\cH}$ denote the standard measure on the
unit tangent bundle of this horizontal leaf, i.e. 
$$d\lambda_{\cH} = \sqrt{\det h(x)} dx^1\ldots dx^n \cdot d\sigma_{h,x},    $$
where $d\sigma_{h,x}$ is the standard measure on the unit sphere in $\cH_x$.
The measure on
$U\cH$ can therefore be written locally as
\beq\label{horizontalmeasure}
d\nu_{\cH}
=
d\lambda_{\cH}\,
w(x,y)\,dy,
\qquad
w(x,y):=\sqrt{\det k(x,y)}.
\eeq

The geodesic flow of the horizontal leaf preserves
$d\lambda_{\cH}$. Hence the only contribution to the divergence of
$G_{\cH}$ with respect to $d\nu_{\cH}$ comes from the factor $w$.
Thus
\beq\label{weighteddivergence}
\operatorname{div}_{d\nu_{\cH}}G_{\cH}
=
G_{\cH}(\log w).
\eeq
As the function $\log w$ at a point $(p,X)\in U\cH$ only depends on the point $p\in M$, we see that  
\beq\label{weighteddivergence}
\operatorname{div}_{d\nu_{\cH}}G_{\cH}|_{(p,X)}
=
G_{\cH}|_{(p,X)}(\log w) = X(\log w).
\eeq

We now see how this derivative relates to the mean curvature of the vertical
leaves. Put
\[
Y_\alpha=\frac{\partial}{\partial y^\alpha},
\qquad
X_i=\frac{\partial}{\partial x^i}.
\]
We have 
\beq\label{Xlogw}
X(\log w)
=
\frac12 \sum_{\a,\b=1}^n k^{\alpha\beta}X(k_{\alpha\beta}).
\eeq
Now
\[
k_{\alpha\beta}=g(Y_\alpha,Y_\beta),
\]
and therefore, using the metric compatibility of $\n$,
\[
X_i(k_{\alpha\beta})
=
g(\n_XY_\alpha,Y_\beta)
+
g(Y_\alpha,\n_XY_\beta).
\]
Since the coordinates are adapted to the two foliations,  $[X_i,Y_\alpha]=0$ 
and moreover $\n_{X_i}Y_\alpha=\n_{Y_\alpha}X_i$.  As $X_i$ is horizontal, we obtain
\begin{align*}
X_i(k_{\alpha\beta})
&=
g(\n_{Y_\alpha}X_i,Y_\beta)
+
g(Y_\alpha,\n_{Y_\beta}X_i)\\
&=
-g(X_i,\n_{Y_\alpha}Y_\beta)
-g(X_i,\n_{Y_\beta}Y_\alpha).
\end{align*}
Therefore
\beq\label{Xlow2}
\frac12 k^{\alpha\beta}X_i(k_{\alpha\beta})
=
-k^{\alpha\beta}
g(\n_{Y_\alpha}Y_\beta,X_i).
\eeq
Since the mean curvature vector of the leaves is $N=
\sum_{\a,\b=1}^nk^{\alpha\beta}
(\n_{Y_\alpha}Y_\beta)^{\cH}$, by \eqref{weighteddivergence}, \eqref{Xlogw} and \eqref{Xlow2} we get 
$$
X_i(\log w)=-g(N,X_i),\qquad i=1,\ldots,n,$$
hence our claim. \end{proof}
Note that the identity $X(\log w)=-g(N,X)$ is the local expression of the classical formula of Rummler relating
the transverse variation of the leafwise volume form to the mean
curvature form (see \cite{Rummler}).

\begin{remark}
The tangent sphere bundle of a foliation and its leafwise geodesic
flow have been studied by Rovenski and Walczak
\cite{RovenskiWalczak}. In particular, for a totally geodesic
foliation they express the Jacobian of this flow in terms of the
transverse Jacobi tensor.  In our situation, applied to the horizontal
foliation, the first variation of their formula gives
\[
\operatorname{div}_{d\nu_{\cH}}G_{\cH}(p,X)
=\operatorname{tr}(Y\mapsto(\n_YX)^\cV)
=-g(N,X).
\]
We have included the above direct computation for completeness.
\end{remark}

\begin{remark}
The same defect can be seen directly on $M$. If $X$ is horizontal and
$\{X_i\}$, $\{V_\alpha\}$ are local orthonormal frames of $\cH$ and
$\cV$, respectively, then
\begin{align*}
\operatorname{div}_M X
&=
\sum_i g(\n_{X_i}X,X_i)
+
\sum_\alpha g(\n_{V_\alpha}X,V_\alpha)\\
&=
\operatorname{div}_{\cH}X
-
\sum_\alpha g(X,\n_{V_\alpha}V_\alpha)\\
&=
\operatorname{div}_{\cH}X-g(X,N).
\end{align*}
Thus the mean curvature $N$ measures exactly the failure of the
ambient measure to be invariant under the horizontal geodesic flow.
\end{remark}

If $N=0$, Lemma~\ref{horizontalLiouville} shows that $G_{\cH}$
preserves $d\nu_{\cH}$. We therefore obtain a horizontal analogue of
Ros' formula.

\begin{proposition}\label{horizontalRos}
Assume that the leaves of $\cF$ are minimal. Let $Q$ be a covariant
$m$-tensor on $\cV$. Then
\[
\int_{U\cH}
(\bar\nabla_vQ)(Jv,\ldots,Jv)\,d\nu_{\cH}=0.
\]
\end{proposition}

\begin{proof}
Let $G_{\cH}$ denote the horizontal geodesic vector field on $U\cH$.
By the preceding discussion, minimality implies that $G_{\cH}$
preserves the natural measure $d\nu_{\cH}$.

Let $\gamma$ be a horizontal geodesic and put $v=\dot\gamma$. Since
$\cH$ is totally geodesic, $\bar\nabla_vv=\nabla_vv=0$. 
Moreover, since $\bar\nabla J=0$,
then 
$\bar\nabla_v(Jv)=0$.
Hence, if
\[
\widehat Q(p,v):=Q_p(Jv,\ldots,Jv),
\]
then
\[
G_{\cH}\widehat Q
=
(\bar\nabla_vQ)(Jv,\ldots,Jv).
\]
Since $G_{\cH}$ preserves $d\nu_{\cH}$ and $U\cH$ is compact,
\[
0
=
\int_{U\cH}G_{\cH}\widehat Q\,d\nu_{\cH}
=
\int_{U\cH}
(\bar\nabla_vQ)(Jv,\ldots,Jv)\,d\nu_{\cH}.
\]
\end{proof}

\par\bigskip

\section{The proof ot the main Theorem.}

From now on we assume that all leaves of $\mathcal F$ are minimal, i.e. 
\beq\label{standingassumptions}
N=0.
\eeq
Note that in particular,
\beq\label{traceSminimal}
\tr S=0.
\eeq

As we like to understand the variation of $S$ along the horizontal leaves, we introduce the covariant tensor $C$ on $\cV$ defined by
\beq\label{definitionC}
C(U,V,W,Z)
:=
(\bar\n_{JU}S)(V,W,Z).
\eeq

We first prove two properties of $C$.

\begin{lemma}\label{symmetryC}
Under the assumption \eqref{standingassumptions}, one has
\[
C\in S^4(\cV^*),
\qquad
\tr C=0.
\]
\end{lemma}

\begin{proof}
Since $S$ is symmetric, $C(U,V,W,Z)$ is automatically symmetric in
$V,W,Z$. It remains to prove symmetry in the first two entries.

For $U,V,W,Z\in\cV$, we take $X=JU, Y=JZ$ in formula \eqref{mixed-curvature} and we obtain 
\[
g(R(JU,V)W,JZ)
=
g((\bar\n_{JU}T)_VW,JZ)
+
g(T_{T_V(JU)}W,JZ).
\]
The first term on the right hand side is $C(U,V,W,Z)$. Moreover, we recall that $T_V(JU)=-B_VU=-B_UV$, 
where in the last equality we have used the total symmetry of $S$.
Hence
\[
g(T_{T_V(JU)}W,JZ) = -g(JT_WZ,JT_VU) 
=
-S(B_UV,W,Z).
\]
Therefore
\beq\label{Ccurvature}
C(U,V,W,Z)
=
g(R(JU,V)W,JZ)
+
S(B_UV,W,Z).
\eeq

Interchanging $U$ and $V$, the second term on the right hand side is
unchanged. On the other hand, the K\"ahler curvature identities give
\[
R(JV,U)
=
-R(V,JU)
=
R(JU,V).
\]
Thus
\[
C(U,V,W,Z)=C(V,U,W,Z),
\]
and therefore $C$ is totally symmetric.\par 
As $N=0$ and $\Ric=0$, the formula \eqref{ricci-main} for the Ricci tensor reduces to  
\[
\Ric(U,V)
=
\sum_i(\bar\n_{JV_i}S)(U,V,V_i) =0.
\]
Therefore
\[
0
=
\sum_iC(V_i,U,V,V_i)
\]
and, since $C$ is totally symmetric, this is precisely our claim $\tr C=0$.\end{proof}

\par\medskip

We now apply the horizontal Ros' formula, that becomes meaningful only when the degree of the tensor $Q$ is odd. Let $Q$ be the symmetric
covariant tensor of degree $7$ obtained by symmetrizing $S\otimes C$. For semplicity we may suppose that for every unit tangent vector $v\in \cV$ 
\[
Q(v,\ldots,v)
=
S(v,v,v)C(v,v,v,v).
\]

Define the covariant tensor $D$ of degree $5$ on $\cV$ by
\beq\label{definitionD}
D(E,L,U,V,W)
:=
(\bar\n_{JE}C)(L,U,V,W).
\eeq
Then
\[
(\bar\n_{Jv}Q)(v,\ldots,v)
=
C(v,v,v,v)^2
+
S(v,v,v)D(v,v,v,v,v).
\]
If $v\in U\cH$ and we put $u=Jv\in U\cV$, then
$v=-Ju$. Hence Proposition~\ref{horizontalRos}, together with the
preceding identity and multiplication by $-1$, gives
\beq\label{horizontalRosSC}
0=
\int_{U\cH}
\left[
C(Jv,Jv,Jv,Jv)^2
+
S(Jv,Jv,Jv)D(Jv,Jv,Jv,Jv,Jv)
\right]d\nu.
\eeq

\par\medskip

We next perform the integration on the fibers. We use the standard moment formula for the uniform measure on the
sphere (see \cite[Sec.~3]{VignatBhatnagar} and also \cite{Folland}, \cite{Efthimiou}):
\beq\label{sphereint}
\int_{S^{n-1}}x_{i_1}\cdots x_{i_{2m}}\,d\sigma
=
\frac{\vol(S^{n-1})}
{n(n+2)\cdots(n+2m-2)}
\sum_{\mathcal P}
\prod_{\{a,b\}\in\mathcal P}\delta_{i_ai_b},
\eeq
where $\mathcal P$ is the set of all $(2m-1)!! = 1 \cdot 3 \cdot 5 \cdots (2m-1)$ distinct ways to pair the $2m$ indices 
$\{1,\ldots,2m\}$.
In our case $m=4$ and we put $c_n:= \frac{\vol(S^{n-1})}{n(n+2)(n+4)(n+6)}$.
Since $C$ is symmetric and trace-free, the only nonzero pairings in
the integral of $C(v,v,v,v)^2$ are those where every index of the
first copy of $C$ is paired with an index of the second copy of $C$. There are $4!$
such pairings, and therefore
\beq\label{integralC}
\int_{U\cH_p}C(Jv,Jv,Jv,Jv)^2\,d\sigma_p(v)
=
24c_n|C|^2.
\eeq

We now consider the mixed term. Fix an orthonormal basis
$\{u_i\}$ of $\cH_p$ and write
\[
S_{abc}:=S(Ju_a,Ju_b,Ju_c),
\qquad
D_{defgh}:=D(Ju_d,Ju_e,Ju_f,Ju_g,Ju_h).
\]
The tensor $D$ is symmetric in its last four indices. Moreover, since
$\tr C=0$ and $\bar\n$ is metric, $D$ is trace-free in any two of its
last four indices.

Since $S$ is trace-free, every nonzero pairing in
\[
S_{abc}D_{defgh}
\int_{S^{n-1}}
x_ax_bx_cx_dx_ex_fx_gx_h\,d\sigma
\]
must pair each of the three indices $a,b,c$ with an index of $D$.

If the first index $d$ of $D$ is paired with one of $a,b,c$, then two
among the last four indices of $D$ must be paired with each other,
and that contribution vanishes because $D$ is trace-free in its last
four indices.

Thus, in every surviving pairing, $d$ is paired with one among
$e,f,g,h$, and the three indices $a,b,c$ are paired with the remaining
three indices of $D$. There are
\[
4\cdot3!=24
\]
such pairings. Therefore, using the fact that $D$ is also symmetric in the last four entries, we have 
\beq\label{mixedfiberintegral}
\int_{U\cH_p}
S(Jv,Jv,Jv)D(Jv,Jv,Jv,Jv,Jv)\,d\sigma_p(v)
=
24c_n\langle S,H\rangle,
\eeq
where the symmetric cubic tensor $H$ on $\cV$ is defined by
\beq\label{definitionH}
H(U,V,W)
:=
\sum_{i=1}^nD(V_i,V_i,U,V,W).
\eeq

Combining \eqref{horizontalRosSC}, \eqref{integralC} and
\eqref{mixedfiberintegral}, we obtain
\beq\label{Rosreduced}
0=
24c_n
\int_M
\left(
|C|^2+\langle S,H\rangle
\right)d\mu_M.
\eeq

\par\medskip

We now compute $H$. 
We recall that 
\[
\sum_iD(U,V_i,V_i,V,W)
=
\sum_i(\bar\n_{JU}C)(V_i,V_i,V,W)
=
0
\]
because $\tr C=0$. Hence
\beq\label{Hantisym}
H(U,V,W)
=
\sum_i
\left[
D(V_i,U,V_i,V,W)
-
D(U,V_i,V_i,V,W)
\right].
\eeq

Therefore we need to compute the antisymmetric part of $D$ in its first two
entries (this point requires some care because $\bar\n$ has torsion).

For $E,L,U,V,W\in\cV$, using $\bar\n J=0$, one has
\[
D(E,L,U,V,W)
=
\left(
\bar\n_{JE}\bar\n_{JL}S
-
\bar\n_{\bar\n_{JE}(JL)}S
\right)(U,V,W).
\]
Therefore
\begin{align*}
&D(E,L,U,V,W)-D(L,E,U,V,W)\\
&=
\left(
\bar R(JE,JL)\cdot S
-
\bar\n_{\tor^{\bar\n}(JE,JL)}S
\right)(U,V,W).
\end{align*}
Now $JE,JL$ are horizontal. Since $A=0$, we have
$\bar\n_{JE}JL=\n_{JE}JL$ and
$\bar\n_{JL}JE=\n_{JL}JE$. Hence, since $\n$ is torsion-free,
\[
\tor^{\bar\n}(JE,JL)
=
\bar\n_{JE}JL-\bar\n_{JL}JE-[JE,JL]
=
\n_{JE}JL-\n_{JL}JE-[JE,JL]
=
0.
\]
It follows that
\begin{align*}
&D(E,L,U,V,W)-D(L,E,U,V,W)\\
&=
-S(\bar R(JE,JL)U,V,W)
-S(U,\bar R(JE,JL)V,W)\\
&\qquad
-S(U,V,\bar R(JE,JL)W).
\end{align*}
We note that the horizontal foliation is totally geodesic and therefore for horizontal directions, $\bar R$ coincides with $R$ on $\cV$.
Using the K\"ahler curvature identity and \eqref{f4},
\[
\bar R(JE,JL)|_{\cV}
=
R(JE,JL)|_{\cV}
=
R(E,L)|_{\cV}
=
-[B_E,B_L].
\]
It then follows that 
\beq\label{Dcommutator}
\begin{aligned}
&D(E,L,U,V,W)-D(L,E,U,V,W)\\
&=
S([B_E,B_L]U,V,W)
+
S(U,[B_E,B_L]V,W)
+
S(U,V,[B_E,B_L]W).
\end{aligned}
\eeq

Substituting \eqref{Dcommutator} into \eqref{Hantisym}, and writing
$B_i:=B_{V_i}$, we get
\beq\label{Hcommutator}
\begin{aligned}
H(U,V,W)
&=
\sum_{i=1}^n\{
S([B_i,B_U]V_i,V,W)\\
&+
S(V_i,[B_i,B_U]V,W) +
S(V_i,V,[B_i,B_U]W)\}.
\end{aligned}
\eeq

\par\medskip

We now write
\[
\langle S,H\rangle=I_1+I_2+I_3
\]
according to the three terms in \eqref{Hcommutator}.\par\medskip

\noindent $\bullet$\ As for $I_1$, we first recall that by minimality $\sum_{i=1}^nB_iV_i=0$, so that 
$$\sum_{i=1}^n[B_i,B_a]V_i=\sum_{i=1}^nB_iB_aV_i-B_a\sum_{i=1}^nB_iV_i=\sum_{i=1}^nB_i^2V_a,$$
where we have used the fact that $B_aV_i=B_iV_a$. If we put
\beq\label{definitionP}
P:=\sum_iB_i^2.
\eeq
then 
\begin{align*}
I_1
&=
\sum_{a,b,c,i}
S_{abc}
S([B_i,B_a]V_i,V_b,V_c) = \sum_{a,b,c}S_{abc}
S(PV_a,V_b,V_c)\\
&= \sum_{a,b,c}g(B_bV_a,V_c)g(B_bV_c,PV_a) = \sum_{a,b,c} g(B_bV_a,V_c)g(V_c,B_bPV_a)\\
&= \sum_{a,b}g(B_bV_a,B_bPV_a) = \sum_{a,b}g(V_a,B_b^2PV_a) = |P|^2.
\end{align*}
\noindent $\bullet$\ As for $I_2$, put
\[
A_{ia}:=[B_i,B_a].
\]
We have 
$$I_2=
\sum_{i,a,b,c}S_{abc}S(V_i,A_{ia}V_b,V_c)=\sum_{i,a}\tr(B_aB_iA_{ia}).$$

Now note that if $K,L$ are two symmetric endomorphisms and $M=[K,L]$, then taking the transpose in the trace we have
$$T:= \Tr(LKM) = -\Tr(MKL) = -\Tr(KLM).$$
This implies that 
$$\Tr((LK-KL)M)=2T = -\Tr(M^2)$$
and therefore 
$$I_2 = \frac 12 \sum_{i,a}||[B_i,B_a]||^2.$$
\noindent $\bullet$ For the last term, we have 
$$
I_3=
\sum_{i,a,b,c}
S_{abc}
S(V_i,V_b,A_{ia}V_c)=
-\sum_{i,a}\tr(B_aA_{ia}B_i).$$
Using again cyclicity of the trace and the same elementary matrix
identity, we obtain
\beq\label{I3}
I_3
=
\frac12\sum_{i,a}|[B_i,B_a]|^2.
\eeq

Therefore we obtain 
\beq\label{SHpositive}
\langle S,H\rangle
=
\left|\sum_iB_i^2\right|^2
+
\sum_{i,a}|[B_i,B_a]|^2.
\eeq

Substituting \eqref{SHpositive} into \eqref{Rosreduced}, we finally
obtain
\beq\label{finalpositiveidentity}
0=
24c_n
\int_M
\left[
|C|^2
+
\left|\sum_iB_i^2\right|^2
+
\sum_{i,a}|[B_i,B_a]|^2
\right]d\mu_M.
\eeq
As all three terms are nonnegative, we obtain 
\[
C=0,
\qquad
\sum_iB_i^2=0,
\qquad
[B_i,B_a]=0.
\]
Since every $B_i$ is symmetric, $\tr\left(\sum_iB_i^2\right)
=\sum_i|B_i|^2  =0$, which forces $B_i=0$ for every $i=1,\ldots, n$, hence $T\equiv 0$. Together with $A\equiv 0$, this implies that the distributions $\cV$ and $\cH$ are parallel, hence the mixed curvature vanishes. Since the vertical leaves are flat, the K\"ahler identities imply that $R\equiv 0$ and our claim is proved.

\end{document}